\documentclass[10pt]{article}
\usepackage{amsmath, amsthm}
\usepackage{formula}

\newtheoremstyle{theorem}
{10pt} 
{10pt} 
{\sl} 
{\parindent} 
{\bf} 
{. } 
{ } 
{} 
\theoremstyle{theorem}
\newtheorem{theorem}{Theorem}

\newtheoremstyle{defi}
{10pt} 
{10pt} 
{\rm} 
{\parindent} 
{\bf} 
{. } 
{ } 
{} 
\theoremstyle{defi}

\begin{document}

\title{\vspace*{35mm}ASYMPTOTIC RAYS}
\author{Oleksii Kuchaiev$^1$ and Anastasiia Tsvietkova$^2$\\
$^1$Department of Computer Science, \\ Donald Bren School of Information and Computer Sciences, \\ University of
California, Irvine, \\ 3019 Donald Bren Hall, Irvine, \\ CA
92697-3435, USA\\
oleksii.kuchaiev@uci.edu\\[2pt]
$^2$Department of Mathematics, \\ The
University of Tennessee,  \\ 104 Aconda Court, 1534 Cumberland
Avenue, \\ Knoxville, TN 37996-0612, USA\\
tsvietkova@math.utk.edu}

\maketitle

\begin{abstract} We prove
that a graph $\Gamma$ is asymptotically isomorphic to the ray if
and only if $\Gamma$ is uniformly spherically bounded and is of
bounded local degrees. This problem arouse in combinatorics and was posed
in \cite{ballstruct} (Problem 10.1).
\


\end{abstract}

\thispagestyle{empty}

\

A ray $\mathcal{R}$ is a non-oriented graph with the set of
vertices $ \omega=\{0,1,...\}\ $  and the set of edges $
\{(i,i+1):i\in\omega\}$. An asymptotic ray is a non-oriented
graph, asymptotically isomorphic to the ray. The notion of
asymptotic isomorphism arouse from the following general
combinatoric scheme.

A \textit{ball structure} is a triple $\mathcal{B}=(X,P,B)$ where
$X,P$ are non-empty sets, and for all $ x\in X $ and $ \alpha\in P
$ , $ B(x,\alpha) $ is a subset of $X$ which is called a ball of
radius $\alpha$ around $x$. It is supposed that $x \in
B(x,\alpha)$ for all $x \in X$, $ \alpha \in P$. The set $X$ is
called the \textit{support} of $\mathcal{B}$, $P$ is called the
\textit{set of radiuses}.

Given any $x \in X , A\subseteq X, \alpha \in P$, we put
\begin{eqnarray} \nonumber B^{*}(x,\alpha)=\{y \in X : x \in B(y,\alpha)\}, \
\displaystyle B(A,\alpha)=\bigcup_{a \in A}B(a,\alpha)  . \
\end{eqnarray}

A ball structure $\mathcal{B}=(X,P,B)$ is called
\begin{itemize} \item \textit{lower symmetric} if, $\forall
\alpha, \beta \in P \ \exists \alpha', \beta' \in P$ such
that$\forall x \in X$
\begin{eqnarray}
\nonumber  B^{*}(x,\alpha')\subseteq B(x,\alpha), \
B(x,\beta')\subseteq
 B^{*}(x,\beta);
\end{eqnarray}
\item \textit{upper symmetric} if $\forall \alpha, \beta \in P \
\exists \alpha', \beta' \in P$ such that $\forall x \in X$
\begin{eqnarray}
 \nonumber B(x,\alpha)\subseteq B^{*}(x,\alpha'), \ B^{*}(x,\beta)\subseteq
 B(x,\beta');
\end{eqnarray}
\item \textit{lower multiplicative} if $\forall \alpha,\beta \in
P\  \exists \gamma \in P$ such that $\forall x \in X$
\begin{eqnarray}
\nonumber B(B(x,\gamma),\gamma)\subseteq B(x,\alpha)\cap
B(x,\beta);
\end{eqnarray}
\item \textit{upper multiplicative} if $\forall \alpha,\beta \in
P\  \exists \gamma \in P$ such that $\forall x \in X$
\begin{eqnarray}
\nonumber B(B(x,\alpha),\beta)\subseteq B(x,\gamma);
\end{eqnarray}
\end{itemize} \qquad
Let $\mathcal{B}=(X,P,B)$ be a lower symmetric and lower
multiplicative ball structure. Then the family \
\begin{eqnarray}
\nonumber \{\displaystyle \bigcup_{x \in X}{B(x,\alpha)\times
B(x,\alpha): \alpha \in P}\}
\end{eqnarray}
is a base of entourages for some (uniquely determined) uniformity
on $X$. For information about uniformity and uniform topological
spaces see \cite{burbaki}. On the other hand, if $\textit{U} \subseteq X
\times X$ is a uniformity on $X$, then the ball structure
$(X,\textit{U},B)$ is lower symmetric and lower multiplicative,
where $B(x,U)=\{y \in X:(x,y)\in U\}$. Thus, the lower symmetric
and lower multiplicative ball structures can be naturally
identified with  uniform topological spaces.\ \

\thispagestyle{empty}

A ball structure which is upper symmetric and upper multiplicative
is called \textit{ballean}. The balleans appeared independently in
asymptotic topology \cite{dranish} under name coarse structures and in
combinatorics \cite{ballstruct}. Directly from the definition it follows that
the balleans can be considered as asymptotic counterparts of
uniform topological spaces. For more details about this duality
see \cite{dranish,ballstruct}. The role of morphisms in the category of uniform
topological spaces is played by uniformly continuous mappings, and
that is why it is necessary to define its asymptotic equivalents.\

Let $\mathcal{B}_{1}=(X_{1},P_{1},B_{1})$ and
$\mathcal{B}_{2}=(X_{2},P_{2},B_{2})$ be balleans.\

A mapping $f:X_{1}\rightarrow X_{2}$ is called a
\textit{$\prec$-mapping} if $\forall \alpha \in P_{1}\ \exists
\beta \in P_{2}$ such that:
\begin{eqnarray}
\nonumber f(B_{1}(x,\alpha))\subseteq B_{2}(f(x),\beta).
\end{eqnarray}\ \qquad
A bijection $f:X_{1}\rightarrow X_{2}$ is called an
\textit{asymptotic isomorphism} (briefly \textit{asymorphism})
between $\mathcal{B}_{1}$ and $\mathcal{B}_{2}$ if $f$ and
$f^{-1}$ are $\prec$-mappings, and $\mathcal{B}_{1}$ and
$\mathcal{B}_{2}$ are called asymorphic.\

For an arbitrary ballean $\mathcal{B}=(X,P,B)$ a family $\Im$ of
subsets of the support $X$ is called \textit{uniformly bounded},
if $\exists \alpha \in P$ such that $\forall$ $F\in \Im,
F\subseteq B(x,\alpha)$ for some $x\in X$. A bijection
$f:X_{1}\rightarrow X_{2}$ is an asymorphism between
$\mathcal{B}_{1}$ and $\mathcal{B}_{2}$ iff for any uniformly
bounded family $\Im$ of subsets of $X_{1}$, the family
$f(\Im)=\{f(F): F\in \Im\}$ is uniformly bounded in
$\mathcal{B}_{2}$, and for any uniformly bounded family $\Im'$of
subsets of $X_{2}$, the family $f^{-1}(\Im')=\{f^{-1}(F): F\in
\Im' \}$ is uniformly bounded in $\mathcal{B}_{1}$.\

Every metric space $(X,d)$ determines the ballean
$\mathcal{B}(X,d)=(X,\mathbb{R^{+}},B_{d})$, where
$\mathbb{R^{+}}$ is the set of non-negative real numbers, \

$B_{d}(x,r)=\{y\in X: d(x,y)\leq r \}$.\

A ballean $\mathcal{B}$ is called \textit{metrizable} if
$\mathcal{B}$ is asymorphic to $\mathcal{B}(X,d)$ for some metric
space $(X,d)$. A criterion of metrizability of a ballean can be
found in \cite{ballstruct} (Theorem 9.1).

Every connected graph $\Gamma (V,E)$ with the set of vertices $V$
and the set of edges $E$ determines the metric space $(V,d)$,
where $d(u,v)$ is the length of the shortest path from $u$ to $v$.
Thus, for every graph $\Gamma$ there is a ballean
$\mathcal{B}(\Gamma)=\mathcal{B}(V,d)$ with the support $V$, which
corresponds to $\Gamma$. A ballean $\mathcal{B}$ is called
\textit{a graph ballean} if $\mathcal{B}$ is asymorphic to the
ballean $\mathcal{B}(\Gamma)$ of some connected graph $\Gamma$. A
criterion of the graph ballean can be found in \cite{ballstruct} (Theorem 9.2).
In what follows we will consider only connected graphs.\

The following lemma makes clear  the notion of $\prec$-mapping for
graph balleans.

\thispagestyle{empty}

\textbf{Lemma 1}.\textit{ Let $\Gamma_{1}(V_{1},E_{1})$ and
$\Gamma_{2}(V_{2},E_{2})$ be connected graphs. Then the following
statements are equivalent:
\begin{description}
\item[$(i)$] $f$ is a $\prec$-mapping between
$\mathcal{B}(\Gamma_{1})$ and $\mathcal{B}(\Gamma_{2})$;
\item[$(ii)$] there exists a natural number $m$, such that
$f(B_{1}(v,1))\subseteq B_{2}(f(v),m)$ for every vertex $v\in
V_{1}$, where $B_{1}$ and $B_{2}$ are balls of corresponding
radiuses in $\Gamma_{1}$ and $\Gamma_{2}$; \item[$(iii)$] there
exists a natural number $m$, such that $d_{2}(f(v),f(u))\leq
md_{1}(v,u)$ for any $v,u \in V_{1}$, where $d_{1}$ and $d_{2}$
are metrics on $V_{1}$ and $V_{2}$.
\end{description}}
\begin{proof}
$(i)\Rightarrow (ii)$ follows directly from the
definition of $\prec$-mapping. \

$(ii)\Rightarrow (iii)$ If $d(v,u)=1$, then $u\in B_{1}(v,1)$ so
$d_{2}(f(v),f(u))\leq m$. For any $v,u\in V_{1}$ we choose the
shortest path $v=v_{0},v_{1},...,v_{k}=u$ between $u$ and $v$.
Since $d_{2}(f(v_{i}),f(v_{i+1}))\leq m$ for every $i\in
\{0,...,k-1\}$, we have $d_{2}(f(v),f(u))\leq m\cdot
k=md_{1}(v,u)$.

$(iii)\Rightarrow (i)$. It suffices to notice that $(iii)$ is
equivalent to: $f(B_{1}(v,k))\subseteq B_{2}(f(v),m\cdot k)$ for
every $v\in V_{1}$.
\end{proof}

Thus, $\prec$-mappings of graph balleans is the Lipschitz mappings
between the metric spaces of the corresponding graphs. \

A graph $\Gamma (V,E)$ is called \textit{bounded} if there exists
a natural number $m$ such that $d(u,v)\leq m$ for all $u,v\in V$.
If $\Gamma_{1}(V_{1},E_{1})$, $\Gamma_{2}(V_{2},E_{2})$ are graphs
and $\Gamma_{1}$ is bounded, then $\mathcal{B}(\Gamma_{1})$ is
asymorphic with $\mathcal{B}(\Gamma_{2})$ iff $\Gamma_{2}$ is
bounded and $|V_{1}|=|V_{2}|$. Hence, the problem of asymorphism
between graph balleans concerns unbounded graphs only. \

\thispagestyle{empty}

The simplest example of unbounded graph is a ray - a non-oriented
graph $\mathcal{R}$ with a set the vertices $ \omega=\{0,1,...\}\
$ and the set of edges $ \{(i,i+1):i\in\omega\}$. We say that a
graph $\Gamma$ is an \textit{asymptotic ray} if the balleans
$\mathcal{B}(\Gamma)$ and $\mathcal{B}(\mathcal{R})$ are
asymorphic. Remind that the degree $\rho(v)$ of a vertex $v$ of a
graph $\Gamma$ is the number of edges incident to $v$.\

\textbf{Lemma2}.\textit{ Let $\Gamma (V,E)$ be an asymptotic ray.
Then there exists a natural number $m$, such that $\rho(v)\leq m$
for every $v\in V$.}
\begin{proof}. Fix a bijection
$f:V\rightarrow \omega$, which is a $\prec$-mapping between
$\mathcal{B}(\Gamma)$ and $\mathcal{B}(\mathcal{R})$. Choose a
natural number $k$ such that $f(B(v,1))\subseteq B'(f(v),k)$ for
every $v\in V$, where $B$ and $B'$ are balls of corresponding
radiuses in $\Gamma$ and $\mathcal{R}$. It is obvious that
$|B'(u,k)|\leq 2k+1$ for all $u \in \omega$. Since $f$ is a
bijection then $|B(v,1)|\leq 2k+1$ for all $v\in V$, so we can put
$m=2k$.
\end{proof}

Let $\Gamma (V,E)$ be an arbitrary graph. For any $v\in V$ and
$k\in \omega$, we put $S(v,k)=\{u\in V:d(u,v)=k\}$. An injective
sequence of vertices $(v_{n})_{n\in \omega}$ of $\Gamma$ is called
\textit{an arrow}, which starts from the vertex $v_{0}$, if
$(v_{i},v_{i+1})\in E$ and $v_{i}\in S(v_{0},i)$ for all $i\in
\omega$. A graph $\Gamma$ is called locally finite if the degrees
of all its vertices are finite.  By K\o nig lemma, from each
vertex of an infinite locally finite graph starts at least one
arrow. In view of this remark and Lemma2, the following theorem
gives the characterization of asymptotic rays.

\begin{theorem}
Let $\Gamma(V,E)$ be an infinite graph,
$s$ be a natural number such that $\rho(v)\leq s$ for all $v\in
V$, and let $(a_{n})_{n\in \omega}$ be an arrow in $\Gamma$,
$A=\{a_{n}:n\in \omega \}$. Then the following statements are
equivalent:
\begin{description}
\item[$(i)$] $\Gamma$ is an asymptotic ray; \item[$(ii)$] there
exists a natural number $r$ such that $V=B(A,r)$; \item[$(iii)$]
the family $\{S(a_{0},n): n\in \omega \}$ of subsets of $V$ is
uniformly bounded in $B(\Gamma)$.
\end{description}
\end{theorem}
\begin{proof}
$(i)\Rightarrow (ii)$. Fix an asymorphism
$f:V\rightarrow \omega$ between $\mathcal{B}(\Gamma)$ and
$\mathcal{B}(\mathcal{R})$. Using Lemma1, we choose a natural
number $m$ such that $|f(u)-f(v)|\leq m$ for all $(u,v)\in E$.
Consider an injective sequence $(f(a_{n}))_{n\in \omega}$ in
$\omega$ and note that $|f(a_{n+1})-f(a_{n})|\leq m, n\in \omega$.
Put $k=max\{m,\displaystyle \min_{n\in \omega}f(a_{n})\}$ and note
that every segment $[i,i+k],i\in \omega$ contains at least one
element of the sequence $(f(a_{n}))_{n\in \omega}$. Since
$f^{-1}:\omega \rightarrow V$ is a $\prec$-mapping, there exists a
natural number $r$ such that $f^{-1}([i,i+k])\subseteq
B(f^{-1}(i),r)$ for all $i \in \omega$. Since $\omega
=\displaystyle \bigcup_{i\in \omega}[i,i+k]$, $f$ is a bijection
and every segment $[i,i+k]$ contains at least one element of the
sequence $(f(a_{n}))_{n\in \omega}$, then $V=\bigcup
B(a_{n},r)=B(A,r)$.\

$(ii)\Rightarrow (i)$. Fix an arbitrary number $n\in \omega$ and
note that \\ $S(a_{0},n)\bigcap B(a_{k},r)=\emptyset$, if
$|k-n|>r$. Since $V=B(A,r)$, then $S(a_{0},n)\subseteq \bigcup
\{B(a_{k},r):|k-n|\leq r \}\subseteq B(a_{0},2r)$ and the family
$\{S(a_{0},n):n\in \omega \}$ is uniformly bounded.\

$(iii)\Rightarrow (i)$. Define a bijection $f$ between $V$ an
$\omega$ in a such way: put $f(a_{0})=0$, and then numerate in an
arbitrary order the elements of $S(a_{0},1)$, then the elements of
$S(a_{0},2)$ and so on. It  follows clearly from the uniform
boundedness of the family $\{S(a_{0},n): n \in \omega \}$ and
$V=\displaystyle \bigcup_{n\in \omega}S(a_{0},n)$, that $f$ is an
asymorphism.
\end{proof}

In conclusion we precise Theorem1 for trees. Let $(a_{n})_{n\in
\omega}$ be an arrow in some tree $T$. After deletion of edges
(but not vertices) of the arrow, the tree  $T$ disintegrates into
trees $T(a_{n})$ with the roots $a_{n}$,$n\in \omega$.

\begin{theorem}
Let $T(V,E)$ be an infinite tree, $s$
be a natural number such that $\rho(v)\leq s$, for all $v\in V$,
and let $(a_{n})_{n\in \omega}$ be an arrow in T, $A=\{a_{n}:n\in
\omega \}$. The tree $T$ is an asymptotic ray iff there exists a
natural number $t$, such that $|V(T(a_{n}))|\leq t$ for all $n\in
\omega$, where $V(T(a_{n}))$ is the set of vertices of $T(a_{n})$.
\end{theorem}
\begin{proof}
Let $T$ be an asymptotic ray. Using Theorem1,
choose a natural number $r$ such that $V=B(A,r)$. Since $T$ is a
tree, $T(a_{n})\subseteq B(a_{n},r)$ for all $n\in \omega$. Since
$\rho(v)\leq s$ for all  $v\in V$, then $|B(a_{n},r)|\leq
s^{r}+1$. If $|V(T(a_{n}))|\leq t$ for all $v\in \omega$, then
$T(a_{n})\subseteq B(a_{n},t)$, $V=B(A,t)$ and we can use
Theorem1.
\end{proof}

\thispagestyle{empty}

\begin{center} \section*{Acknowledgements} \end{center}  We would like to
acknowledge Prof. Ihor Protasov from the National Taras Shevchenko
University of Kiev for guiding this research.

\

\begin{center} 
\end{center}
\end{document}